\documentclass[12pt]{amsart}
\usepackage{amsmath}
\usepackage{amsthm}
\usepackage{amssymb}
\usepackage{amsfonts}
\usepackage{amsrefs}
\usepackage{color}
\usepackage{tikz}
\tikzstyle{shaded}=[fill=red!10!blue!20!gray!30!white]
\tikzstyle{shaded line}=[double=red!10!blue!20!gray!30!white, double distance=1.5mm, draw=black]
\tikzstyle{unshaded}=[fill=white]
\tikzstyle{unshaded line}=[double=white, double distance=1.5mm, draw=black]
\tikzstyle{Tbox}=[circle, draw, thick, fill=white, opaque,]
\tikzstyle{empty box}=[circle, draw, thick, fill=white, opaque, inner sep=2mm]
\tikzstyle{background rectangle}= [fill=red!10!blue!20!gray!40!white,rounded corners=2mm] 
\tikzstyle{on}=[very thick, red!50!blue!50!black]
\tikzstyle{off}=[gray]
\tikzstyle{traces}=[scale=.2, inner sep=1mm]
\tikzstyle{quadratic}=[scale=.4, inner sep=1mm, baseline]
\tikzstyle{annular}=[scale=.7, inner sep=1mm, baseline]
\tikzstyle{make triple edge size}= [scale=.4, inner sep=1mm,baseline] 
\tikzstyle{icosahedron network}=[scale=.3, inner sep=1mm, baseline]
\tikzstyle{ATLsix}=[scale=.25, baseline]
\tikzstyle{TL12}=[scale=.15,baseline]
\tikzstyle{PAdefn}=[scale=.7,baseline]
\tikzstyle{TLEG}=[scale=.5,baseline]
\makeatletter
\newtheorem{lemma}{Lemma}[section]
\newtheorem{definition}[lemma]{Definition}
\newtheorem{theorem}[lemma]{Theorem}
\newtheorem{proposition}[lemma]{Proposition}
\newtheorem{remark}[lemma]{Remark}
\newtheorem{corollary}[lemma]{Corollary}
\newtheorem{conjecture}[lemma]{Conjecture}
\newtheorem{examples}[lemma]{Examples}
\newtheorem{problem}[lemma]{Problem}
\newtheorem{question}[lemma]{Question}

\newtheorem{notations}[lemma]{Notations}
\newenvironment{claim}[1]{\par\noindent\underline{Claim:}\space#1}{}
\newenvironment{claimproof}[1]{\par\noindent\underline{Proof:}\space#1}{\hfill $ \blacksquare $ }
\makeatother
 \title[Ore's theorem on cyclic subfactor planar algebras]{Ore's theorem on cyclic subfactor planar algebras and beyond}
 \author[Sebastien Palcoux]{Sebastien Palcoux}
\address{Institute of Mathematical Sciences, Chennai, India}
\email{sebastienpalcoux@gmail.com}
\keywords{von Neumann algebra; subfactor; planar algebra; biprojection; distributive lattice; finite group; representation}
 
\begin{document}
\begin{abstract}
Ore proved that a finite group is cyclic if and only if its subgroup lattice is distributive. Now, since every subgroup of a cyclic group is normal, we call a subfactor planar algebra cyclic if all its biprojections are normal and form a distributive lattice. The main result generalizes one side of Ore's theorem and shows that a cyclic subfactor is singly generated in the sense that there is a minimal $ 2 $-box projection generating the identity biprojection. We conjecture that this result holds without assuming the biprojections to be normal, and we show that it is true for small lattices. We finally exhibit a dual version of another theorem of Ore and a non-trivial upper bound for the minimal number of irreducible components for a faithful complex representation of a finite group.
\end{abstract}
\maketitle
\section{Introduction} 
Vaughan Jones proved in \cite{jo2} that the set of possible values for the index $ |M:N| $ of a subfactor $ (N \subseteq M) $ is 
 $$ \{ 4 cos^2(\frac{\pi}{n}) \ \vert \ n \geq 3 \} \sqcup [4,\infty]. $$ We observe that it is the disjoint union of a discrete series and a continuous series.
Moreover, for a given intermediate subfactor $ N \subseteq P \subseteq M $, $ |M:N| = |M:P| \cdot |P:N| $, therefore by applying a kind of Eratosthenes sieve, we get that a subfactor of index in the discrete series or in the interval $ (4,8)$, except the countable set of numbers composed of numbers in the discrete series, can't have a non-trivial intermediate subfactor. A subfactor without non-trivial intermediate subfactor is called maximal \cite{bi}. For example, any subfactor of index in $ (4,3+\sqrt{5}) $ is maximal; (except $ A_{\infty} $) there are exactly $ 19 $ irreducible subfactor planar algebras for this interval (see \cites{jms,amp}), the first example is the Haagerup subfactor \cite{ep}. Thanks to Galois correspondence \cite{nk}, a finite group subfactor, $ (R^G \subseteq R) $ or $ (R \subseteq R \rtimes G) $, is maximal if and only if it is a prime order cyclic group subfactor (i.e. $ G = \mathbb{Z}/p $ with $ p $ prime). Thus we can say that the maximal subfactors are an extension of the prime numbers.
\begin{question}
What could be the extension of the natural numbers?
\end{question}
To answer this question, we need to find a natural class of subfactors, that we will call the ``cyclic subfactors'', satisfying: 
\begin{itemize}
\item[(1)] Every maximal subfactor is cyclic.
\item[(2)] A finite group subfactor is cyclic iff the group is cyclic.
\end{itemize} 
An old and little known theorem published in 1938 by the Norwegian mathematician {\O}ystein Ore states that: 
\begin{theorem}[\cite{or}]
A finite group $ G $ is cyclic if and only if its subgroup lattice $ \mathcal{L}(G) $ is distributive.
\end{theorem}
Firstly, the intermediate subfactor lattice of a maximal subfactor is obviously distributive. Next, by Galois correspondence, the intermediate subfactor lattice of a finite group subfactor is exactly the subgroup lattice (or its reversal) of the group; but distributivity is invariant under reversal, so (1) and (2) hold by Ore's theorem. 
Now an abelian group, and a fortiori a cyclic group, admits only normal subgroups; but T. Teruya generalized in \cite{teru} the notion of normal subgroup by the notion of normal intermediate subfactor, so: 
\begin{definition}
A finite index irreducible subfactor is cyclic if all its intermediate subfactors are normal and form a distributive lattice.
\end{definition}
 
Note that an irreducible finite index subfactor $ (N \subseteq M) $ admits a finite lattice $ \mathcal{L}(N \subseteq M) $ of intermediate subfactors by \cite{wa}, as for the subgroup lattice of a finite group. Moreover, a finite group subfactor remembers the group by \cite{jo}. 
Section \ref{cycl} exhibits plenty of examples of cyclic subfactors: of course the cyclic group subfactors and the (irreducible finite index) maximal subfactors; moreover, up to equivalence, exactly $ 23279 $ among $ 34503 $ inclusions of groups of index $ < 30 $, give a cyclic subfactor. The class of cyclic subfactors is stable under dual, intermediate, free composition and certain tensor products. 
Now the natural problem about cyclic subfactors is to understand in what sense they are ``singly generated''. To answer this question, we extend the following theorem of Ore.
\begin{theorem}[O. Ore, \cite{or}] \label{introre2} If an interval of finite groups $ [H,G] $ is distributive, then $ \exists g \in G $ such that $ \langle H, g \rangle = G $. 
\end{theorem} 
\begin{theorem} \label{introthm}
An irreducible subfactor planar algebra whose biprojections are central and form a distributive lattice, has a minimal $ 2 $-box projection generating the identity biprojection (i.e. w-cyclic subfactor).
\end{theorem} 
\noindent But ``normal'' means ``bicentral'', so a cyclic subfactor planar algebra is w-cyclic. The converse is false, a group subfactor $ (R^G \subseteq R) $ is cyclic if and only if $ G $ is cyclic, and is w-cyclic if and only if $ G $ is linearly primitive (take $ G=S_3 $). That's why we have chosen the name w-cyclic (i.e. weakly cyclic). 
We conjecture that Theorem \ref{introthm} holds without the assumption ``central''.
\begin{conjecture} \label{introconj}
An irreducible subfactor planar algebra with a distributive biprojection lattice is w-cyclic.
\end{conjecture}
\noindent It is true if the lattice has less than $32$ elements (and so, at index $ < 32 $). Now the group-theoretic reformulation of Conjecture \ref{introconj} for the planar algebra $ \mathcal{P}(R^G \subseteq R^H) $, gives a dual version of Theorem \ref{introre2}.
\begin{conjecture} \label{dualore} If the interval of finite groups $ [H,G] $ is distributive then $ \exists V $ irreducible complex representation of $ G $ such that $ G_{(V^H)} = H $. 
\end{conjecture} 
In general, we deduce a non-trivial upper bound for the minimal number of minimal central projections generating the identity biprojection. For $ \mathcal{P}(R^G \subseteq R) $, this gives a non-trivial upper bound for the minimal number of irreducible components for a faithful complex representation of $ G $. It is a bridge linking combinatorics and representations in the theory of finite groups. This paper is a short version of \cite{pa}.
\tableofcontents 
\normalsize
\section{Ore's theorem on finite groups}
\subsection{Basics in lattice theory} \label{baslat}
A lattice $ (L, \wedge , \vee) $ is a poset $ L $ in which every two elements $ a,b $ have a unique supremum (or join) $ a \vee b $ and a unique infimum (or meet) $ a \wedge b $. Let $ G $ be a finite group. The set of subgroups $ K \subseteq G $ forms a lattice, denoted by $ \mathcal{L}(G) $, ordered by $ \subseteq $, with $ K_1 \vee K_2 = \langle K_1,K_2 \rangle $ and $ K_1 \wedge K_2 = K_1 \cap K_2 $. A sublattice of $ (L, \wedge , \vee) $ is a subset $ L' \subseteq L $ such that $ (L', \wedge , \vee) $ is also a lattice. Consider $ a,b \in L $ with $ a \le b $, then the interval $ [a,b] $ is the sublattice $ \{c \in L \ \vert \ a \le c \le b \} $. Any finite lattice admits a minimum and a maximum, denoted by $ \hat{0} $ and $ \hat{1} $. An atom (resp. coatom) is a minimal (resp. maximal) element in $ L \setminus \{\hat{0}\} $ (resp. $ L \setminus \{\hat{1}\} $). The top interval of a finite lattice $ L $ is the interval $ [t,\hat{1}] $, with $ t $ the meet of all the coatoms. The height of a finite lattice $ L $ is the greatest length of a (strict) chain. A lattice is distributive if the join and meet operations distribute over each other. 
\begin{remark} \label{distable}
Distributivity is stable under taking sublattice, reversal, direct product and concatenation.
\end{remark}
\noindent A distributive lattice is called boolean if any element $ b $ admits a unique complement $ b^{\complement} $ (i.e. $ b \wedge b^{\complement} = \hat{0} $ and $ b \vee b^{\complement} = \hat{1} $). The subset lattice of $ \{1,2, \dots, n \} $, with union and intersection, is called the boolean lattice $ \mathcal{B}_n $ of rank $ n $.
Any finite boolean lattice is isomorphic to some $ \mathcal{B}_n $.
\begin{lemma} \label{topBn}
The top interval of a finite distributive lattice is boolean.
\end{lemma} 
\begin{proof}
See \cite[items a-i p254-255]{sta} which uses Birkhoff's representation theorem (a finite lattice is distributive iff it embeds into some $ \mathcal{B}_n $).
\end{proof} 
\noindent A lattice with a boolean top interval will be called \textit{top boolean} (and its reversal, \textit{bottom boolean}).
See \cite{sta} for more details on lattice basics.
\subsection{Ore's theorem on distributive intervals of finite groups}
{\O}ystein Ore proved the following result in \cite[Theorem 4, p267]{or}.
\begin{theorem} \label{ore1}
 
A finite group $ G $ is cyclic if and only if its subgroup lattice $ \mathcal{L}(G) $ is distributive. 
\end{theorem} 
\begin{proof}
 ($ \Leftarrow $): It is just a particular case of Theorem \ref{ore2} with $ H = \{e\} $. \\
 ($ \Rightarrow $): A finite cyclic group $ G = \mathbb{Z}/n $ has exactly one subgroup of order $ d $, denoted by $ \mathbb{Z}/d $, for every divisor $ d $ of $ n $. Now $ \mathbb{Z}/d_1 \vee \mathbb{Z}/d_2 = \mathbb{Z}/lcm(d_1,d_2) $ and $ \mathbb{Z}/d_1 \wedge \mathbb{Z}/d_2 = \mathbb{Z}/gcd(d_1,d_2) $, but lcm and gcd distribute other each over, the result follows.
\end{proof}
 \begin{definition} \label{Hcy}
An interval of finite groups $ [H,G] $ is said to be $ H $-cyclic if there is $ g \in G $ such that $ \langle H,g \rangle = G $. Note that $ \langle H,g \rangle = \langle Hg \rangle $. 
\end{definition}
Ore extended one side of Theorem \ref{ore1} to the interval of finite groups \cite[Theorem 7]{or} for which we will give our own proof (which is a group-theoretic reformulation of the proof of Theorem \ref{centheo}):
\begin{theorem} \label{ore2}
A distributive interval $ [H,G] $ is $ H $-cyclic.
\end{theorem} 
\begin{proof}
The proof follows from the claims below and Lemma \ref{topBn}.
\\
\begin{claim} \label{max} Let $ M $ be a maximal subgroup of $ G $. Then $ [M,G] $ is $ M $-cyclic.
\end{claim}
\begin{claimproof}
For $ g \in G $ with $ g \not \in M $, we have $ \langle M,g \rangle = G $ by maximality.
\end{claimproof}
\\
\begin{claim} \label{preore2}
A boolean interval $ [H,G] $ is $ H $-cyclic.
\end{claim} 
\begin{claimproof}
Let $ M $ be a coatom in $ [H,G] $, and $ M^{\complement} $ be its complement. By the previous claim and induction on the height of the lattice, we can assume $ [H,M] $ and $ [H,M^{\complement}] $ both to be $ H $-cyclic, i.e. there are $ a, b \in G $ such that $ \langle H,a \rangle = M $ and $ \langle H,b \rangle = M^{\complement} $. For $ g=a b $, $ a=g b^{-1} $ and $ b=a^{-1}g $, so $ \langle H,a,g \rangle = \langle H,g,b \rangle = \langle H,a,b \rangle = M \vee M^{\complement} = G $. 
Now, $ \langle H,g \rangle = \langle H,g \rangle \vee H = \langle H,g \rangle \vee (M \wedge M^{\complement}) $ but by distributivity $ \langle H,g \rangle \vee (M \wedge M^{\complement}) = (\langle H,g \rangle \vee M \rangle) \wedge (\langle H,g \rangle \vee M^{\complement} \rangle) $. So $ \langle H,g \rangle = \langle H,a,g \rangle \wedge \langle H,g,b \rangle = G $. The result follows.
\end{claimproof} 
\\
\begin{claim} \label{topred} 
 $ [H , G] $ is $ H $-cyclic if its top interval $ [K,G] $ is $ K $-cyclic.
\end{claim}
\begin{claimproof}
Consider $ g \in G $ with $ \langle K,g \rangle = G $. For any coatom $ M \in [H,G] $, we have $ K \subseteq M $ by definition, and so $ g \not \in M $, then a fortiori $ \langle H,g \rangle \not \subseteq M $. It follows that $ \langle H,g \rangle=G $.
\end{claimproof} 
\end{proof}
\section{Subfactor planar algebras and biprojections}
For the notions of subfactor, subfactor planar algebra and basic properties, we refer to \cites{js,jo4,sk2}. See also \cite[Section 3]{pa} for a short introduction. A subfactor planar algebra is of finite index by definition.
\subsection{Basics on the 2-box space}
Let $ (N \subseteq M) $ be a finite index irreducible subfactor. The $ n $-box spaces $ \mathcal{P}_{n,+} $ and $ \mathcal{P}_{n,-} $ of the planar algebra $ \mathcal{P}=\mathcal{P}(N \subseteq M) $, are $ N' \cap M_{n-1} $ and $ M' \cap M_{n} $. Let $ R(a) $ be the range projection of $ a \in \mathcal{P}_{2,+} $. We define the relations $ a \preceq b $ by $ R(a) \le R(b) $, and $ a \sim b $ by $ R(a) = R(b) $. Let $ e_1:=e^M_N $ and $ id:=e^M_M $ be the Jones and the identity projections in $ \mathcal{P}_{2,+} $. Note that $ tr(e_1) = |M:N|^{-1} = \delta^{-2} $ and $ tr(id) = 1 $. Let $ \mathcal{F}: \mathcal{P}_{2,\pm} \to \mathcal{P}_{2,\mp} $ be the Fourier transform ($ 90^{\circ} $ rotation), $ \overline{a} := \mathcal{F}(\mathcal{F}(a)) $ the contragredient of $ a \in \mathcal{P}_{2,\pm} $, and $ a * b = \mathcal{F}(\mathcal{F}^{-1}(a) \cdot \mathcal{F}^{-1}(b)) $ the coproduct of $ a,b \in \mathcal{P}_{2,\pm} $.

\begin{lemma} \label{th} Let $ a,b,c,d $ be positive operators of $ \mathcal{P}_{2,+} $. Then
\begin{itemize}
\item[(1)] $ a*b $ is also positive, 
\item[(2)] $ [a \preceq b $ and $ c \preceq d] \Rightarrow a*c \preceq b*d $,
\item[(3)] $ a \preceq b \Rightarrow \langle a \rangle \le \langle b \rangle $,
\item[(4)] $ a \sim b \Rightarrow \langle a \rangle = \langle b \rangle $.
\end{itemize} 
\end{lemma}
\begin{proof}
It is precisely \cite[Theorem 4.1 and Lemma 4.8]{li} for (1) and (2). Next, if $ a \preceq b$, then by (2), for any integer $ k $, $ a^{*k} \preceq b^{*k} $, and hence $\forall n $ $$ \sum_{k=1}^n a^{*k} \preceq \sum_{k=1}^n b^{*k}, $$ then $ \langle a \rangle \le \langle b \rangle $ by Definition \ref{gener}. Finally, (4) is immediate from (3).
\end{proof} 
\noindent The next lemma follows by irreducibility (i.e. $ \mathcal{P}_{1,+}=\mathbb{C} $).
\begin{lemma} \label{pre2} \label{2} 
Let $ p,q \in \mathcal{P}_{2,+} $ be projections. Then $$ e_1 \preceq p *\overline{q} \Leftrightarrow pq \neq 0. $$ 
\end{lemma}
\noindent Note that if $ p \in \mathcal{P}_{2,+} $ is a projection then $ \overline{p} $ is also a projection.
\begin{lemma} \label{pfr}
Let $ a,b,c \in \mathcal{P}_{2,+} $ be projections with $ c \preceq a*b $. Then $ \exists a' \preceq c*\overline{b} $ and $ \exists b' \preceq \overline{a}*c $ such that $ a' $, $ b' $ are projections and $ aa', bb' \neq 0 $.
\end{lemma}
\begin{proof} By using Lemmas \ref{th} and \ref{pre2}, and
 $$ e_1 \preceq c*\overline{c} \preceq (a*b)* \overline{c} = a*(b* \overline{c}). $$ 
We can also apply \cite[Lemma 4.10]{li}.
\end{proof}
\subsection{On the biprojections}
\begin{definition}[\cite{li}, Definition 2.14]
A biprojection is a projection $ b \in \mathcal{P}_{2,\pm} $ with $ \mathcal{F}(b) $ a multiple of a projection. 
\end{definition}
Note that $ e_1=e^M_N $ and $ id=e^M_M $ are biprojections.
\begin{theorem}[\cite{bi} p212] \label{bisch}
A projection $ b $ is a biprojection if and only if it is the Jones projection $ e^M_K $ of an intermediate subfactor $ N \subseteq K \subseteq M $. 
\end{theorem}
Therefore the set of biprojections is a lattice of the form $ [e_1,id] $.
 \begin{theorem} \label{biproj} An operator $ b $ is a biprojection if and only if
 $$ e_1 \le b=b^2=b^{\star}=\overline{b} = \lambda b * b, \text{ with } \lambda^{-1} = \delta tr(b). $$ 
\end{theorem}
\begin{proof} See \cite[items 0-3 p191]{la} and \cite[Theorem 4.12]{li}.
\end{proof}
\begin{lemma} \label{prodcoprod} Consider $ a_1,a_2,b \in \mathcal{P}_{2,+} $ with $ b $ a biprojection. Then 
 $$ (b \cdot a_1 \cdot b) * (b \cdot a_2 \cdot b) = b \cdot (a_1 * (b \cdot a_2 \cdot b)) \cdot b = b \cdot ((b \cdot a_1 \cdot b) * a_2) \cdot b, $$ 
 $$ (b*a_1*b) \cdot (b*a_2*b) = b*(a_1 \cdot (b*a_2*b))*b = b*((b*a_1*b) \cdot a_2)*b. $$ 
\end{lemma}
\begin{proof} 
By exchange relations \cite{la} for $ b $ and $ \mathcal{F}(b) $.
 \end{proof}
 
\begin{definition} \label{gener}
Consider $a \in \mathcal{P}_{2,+}$ positive, and let $p_n$ be the range projection of $\sum_{k=1}^n a^{*k}$. By finiteness, there exists $N$ such that for all $m \ge N$, $p_m = p_N$, which is a biprojection \cite[Lemma 4.14]{li}, denoted $\langle a \rangle$, called the biprojection generated by $a$. It is the smallest biprojection $b \succeq a$. For $ S $ a finite set of positive operators, let $ \langle S \rangle $ be the biprojection $\langle \sum_{s \in S}s \rangle$, it is the smallest biprojection $b$ such that $ b \succeq s $, $ \forall s\in S $.
\end{definition}
\subsection{Intermediate planar algebras and $ 2 $-box spaces} \label{interpa} \hspace*{1cm} \\
Let $ N \subseteq K \subseteq M $ be an intermediate subfactor. The planar algebras $ \mathcal{P}(N \subseteq K) $ and $ \mathcal{P}(K \subseteq M) $ can be derived from $ \mathcal{P}(N \subseteq M) $, see \cites{kcb,la0}. 
\begin{theorem} \label{2iso} \label{landau1} \label{landau2}
Consider the intermediate subfactors $$ N \subseteq P \subseteq K \subseteq Q \subseteq M. $$ Then there are two isomorphisms of von Neumann algebras $$ l_K: \mathcal{P}_{2,+}(N \subseteq K) \to e^M_K \mathcal{P}_{2,+}(N \subseteq M) e^M_K, $$ 
 $$ r_K : \mathcal{P}_{2,+} (K \subseteq M) \to e^M_K * \mathcal{P}_{2,+}(N \subseteq M) * e^M_K, $$ for usual $ + $, $ \times $ and $ ()^{\star} $, such that
 $$ l_K(e^K_P) = e^M_P \text{ and } r_K(e^M_Q) = e^M_Q. $$ 
 Moreover, the coproduct $ * $ is also preserved by these maps, but up to a multiplicative constant, $ |M:K|^{1/2} $ for $ l_K $ and $ |K:N|^{-1/2} $ for $ r_K $.
 Then, $ \forall m \in \{l^{\pm 1}_K, r^{\pm 1}_K \} $, $ \forall a_i > 0 $ in the domain of $ m $, $ m(a_i)>0 $ and $$ \langle m(a_1), \dots, m(a_n) \rangle = m (\langle a_1, \dots, a_n \rangle). $$ 
\end{theorem}
\begin{proof} 
Immediate from \cite{kcb} or \cite{la0}, using Lemma \ref{prodcoprod}. \end{proof}

 
\begin{notations} Let $ b_1 \le b \le b_2 $ be the biprojections $ e^M_P \le e^M_K \le e^M_Q $. We define $ l_b:=l_K $ and $ r_b:=r_K $; also $ \mathcal{P}(b_1, b_2):= \mathcal{P}(P \subseteq Q) $ and $$ |b_2:b_1|:=tr(b_2)/tr(b_1)=|Q:P|. $$ 
\end{notations}
\section{Ore's theorem on subfactor planar algebras}
\subsection{The cyclic subfactor planar algebras} \label{cycl} \hspace*{1cm} \\
\noindent In this subsection, we define the class of cyclic subfactor planar algebras, we show that it contains plenty of examples, and we prove that it is stable under dual, intermediate, free composition and certain tensor products. Let $ \mathcal{P} $ be an irreducible subfactor planar algebra.
\begin{definition}[\cite{teru}] \label{normal} 
A biprojection $ b $ is normal if it is bicentral (i.e. $ b $ and $ \mathcal{F}(b) $ are central).
\end{definition} 
\begin{definition}
An irreducible subfactor planar algebra is said to be
\begin{itemize}
\item distributive if its biprojection lattice is distributive.
\item Dedekind if all its biprojections are normal.
\item cyclic if it is both Dedekind and distributive.
\end{itemize}
Moreover, we call a subfactor cyclic if its planar algebra is cyclic.
\end{definition} 
\begin{examples} A group subfactor is cyclic if and only if the group is cyclic; every maximal subfactor is cyclic, in particular every $ 2 $-supertransitive subfactor, as the Haagerup subfactor \cite{asha,izha,ep}, is cyclic. Up to equivalence, exactly $ 23279 $ among $ 34503 $ inclusions of groups of index $ < 30 $, give a cyclic subfactor (more than $ 65\% $).
\end{examples}
\begin{definition} \label{equiv} Let $ G $ be a finite group and $ H $ a subgroup. The core $ H_G $ is the largest normal subgroup of $ G $ contained in $ H $. The subgroup $ H $ is called core-free if $ H_G = \{ 1 \} $; in this case the interval $ [H,G] $ is also called core-free. Two intervals of finite groups $ [A, B] $ and $ [C,D] $ are called equivalent if there is a group isomorphism $ \phi: B/A_B \to D/C_D $ such that $ \phi(A/A_B) = C/C_D $. 
\end{definition}
\begin{remark} A finite group subfactor remembers the group \cite{jo}, but a finite group-subgroup subfactor does not remember the equivalence class of the interval in general. A counterexample was found by V.S. Sunder and V. Kodiyalam \cite{sk}, the intervals $ [\langle (1234) \rangle , S_4] $ and $ [\langle (12)(34) \rangle , S_4] $ are not equivalent whereas their corresponding subfactors are isomorphic; but thanks to the complete characterization \cite{iz} by M. Izumi, it remembers the interval in the maximal case, because the intersection of a core-free maximal subgroup with an abelian normal subgroup is trivial. \end{remark}
\begin{theorem} \label{th2}
 The free composition of irreducible finite index subfactors has no extra intermediate.
\end{theorem}
\begin{proof}
See \cite[Theorem 2.22]{li}. 
\end{proof}
\begin{corollary} \label{corofree}
 The class of finite index irreducible cyclic subfactors is stable under free composition.
\end{corollary}
\begin{proof}
By Theorem \ref{th2}, the intermediate subfactor lattice of a free composition is the concatenation of the lattice of the two components (and see Remark \ref{distable}). By Theorem \ref{landau1} and Lemma \ref{prodcoprod}, the biprojections remain normal.
\end{proof}
The following theorem was proved in the $ 2 $-supertransitive case by Y. Watatani \cite[Proposition 5.1]{wa}. The general case was conjectured by the author, but specified and proved after a discussion with F. Xu.
\begin{theorem} \label{th3}
Let $ (N_i \subset M_i) $, $ i=1,2 $, be irreducible finite index subfactors. Then $$ \mathcal{L}(N_1 \subset M_1) \times \mathcal{L}(N_2 \subset M_2) \subsetneq \mathcal{L}(N_1 \otimes N_2 \subset M_1 \otimes M_2) $$ if and only if there are intermediate subfactors $ N_i \subseteq P_i \subset Q_i \subseteq M_i $, $ i=1,2 $, such that $ (P_i \subset Q_i) $ is depth $ 2 $ and isomorphic to $ (R^{\mathbb{A}_i} \subset R) $, with $ \mathbb{A}_2 \simeq \mathbb{A}_1^{cop} $ the (very simple) Kac algebra $ \mathbb{A}_1 $ with the opposite coproduct.
\end{theorem}
\begin{proof}
Consider the intermediate subfactors $$ N_1 \otimes N_2 \subseteq P_1 \otimes P_2 \subset R \subset Q_1 \otimes Q_2 \subseteq M_1 \otimes M_2 $$ with $ R $ not of tensor product form, $ P_1 \otimes P_2 $ and $ Q_1 \otimes Q_2 $ the closest (below and above resp.) to $ R $ among those of tensor product form. Now using \cite[Proposition 3.5 (2)]{xu}, $ (P_i \subseteq Q_i) $, $ i=1,2 $, are depth $ 2 $, their corresponding Kac algebras, $ \mathbb{A}_i $, $ i=1,2 $, are very simple and $ \mathbb{A}_2 \simeq \mathbb{A}_1^{cop} $ \cite[Definition 3.6 and Proposition 3.10]{xu}. The converse is given by \cite[Theorem 3.14]{xu}.
\end{proof}
\begin{remark} \label{coroprod} By Theorem \ref{th3} and Remark \ref{distable}, the class of (finite index irreducible) cyclic subfactors is stable under certain tensor products (i.e. if there is no cop-isomorphic depth $ 2 $ intermediate), and by Theorem \ref{landau1} and Lemma \ref{prodcoprod}, the biprojections remain normal. \end{remark}
\begin{lemma} \label{c*} If a subfactor is cyclic then the intermediate and dual subfactors are also cyclic.
\end{lemma}
\begin{proof} By Remark \ref{distable}, Theorem \ref{landau1} and Lemma \ref{prodcoprod}. 
\end{proof}
\noindent A subfactor as $ (R \subseteq R \rtimes G) $ or $ (R^G \subseteq R) $ is called a ``group subfactor''. Then, the following lemma justifies the choice of the word ``cyclic''.
\begin{lemma}
A cyclic ``group subfactor" is a ``cyclic group" subfactor. \end{lemma}
\begin{proof}
By Galois correspondence, if a ``group subfactor" is cyclic then the subgroup lattice is distributive, and so the group is cyclic by Ore's Theorem \ref{ore1}. The normal biprojections of a group subfactor corresponds to the normal subgroups \cite{teru}, but every subgroup of a cyclic group is normal.
\end{proof} 
\begin{problem} \label{cykac}
Is a depth $ 2 $ irreducible finite index cyclic subfactor, a cyclic group subfactor?
\end{problem}
The answer could be \textbf{no} because the following fusion ring (discovered by the author \cite{pa210}), the first known to be simple integral and non-trivial, \textit{could be} the Grothendieck ring of a ``maximal" Kac algebra of dimension $ 210 $ and type $ (1,5,5,5,6,7,7) $. 
\begin{scriptsize}
 $$ \begin{smallmatrix}
1 & 0 & 0 & 0& 0& 0& 0 \\
0 & 1 & 0 & 0& 0& 0& 0 \\
0 & 0 & 1 & 0& 0& 0& 0 \\
0 & 0 & 0 & 1& 0& 0& 0 \\
0 & 0 & 0 & 0& 1& 0& 0 \\
0 & 0 & 0 & 0& 0& 1& 0 \\
0 & 0 & 0 & 0& 0& 0& 1 
\end{smallmatrix}, \ 
\begin{smallmatrix}
0 & 1 & 0 & 0& 0& 0& 0 \\
1 & 1 & 0 & 1& 0& 1& 1 \\
0 & 0 & 1 & 0& 1& 1& 1 \\
0 & 1 & 0 & 0& 1& 1& 1 \\
0 & 0 & 1 & 1& 1& 1& 1 \\
0 & 1 & 1 & 1& 1& 1& 1 \\
0 & 1 & 1 & 1& 1& 1& 1 
\end{smallmatrix}, \ 
\begin{smallmatrix}
0 & 0 & 1 & 0& 0& 0& 0 \\
0 & 0 & 1 & 0& 1& 1& 1 \\
1 & 1 & 1 & 0& 0& 1& 1 \\
0 & 0 & 0 & 1& 1& 1& 1 \\
0 & 1 & 0 & 1& 1& 1& 1 \\
0 & 1 & 1 & 1& 1& 1& 1 \\
0 & 1 & 1 & 1& 1& 1& 1 
\end{smallmatrix}, \ 
\begin{smallmatrix}
0 & 0 & 0 & 1& 0& 0& 0 \\
0 & 1 & 0 & 0& 1& 1& 1 \\
0 & 0 & 0 & 1& 1& 1& 1 \\
1 & 0 & 1 & 1& 0& 1& 1 \\
0 & 1 & 1 & 0& 1& 1& 1 \\
0 & 1 & 1 & 1& 1& 1& 1 \\
0 & 1 & 1 & 1& 1& 1& 1 
\end{smallmatrix}, \ 
\begin{smallmatrix}
0 & 0 & 0 & 0& 1& 0& 0 \\
0 & 0 & 1 & 1& 1& 1& 1 \\
0 & 1 & 0 & 1& 1& 1& 1 \\
0 & 1 & 1 & 0& 1& 1& 1 \\
1 & 1 & 1 & 1& 1& 1& 1 \\
0 & 1 & 1 & 1& 1& 2& 1 \\
0 & 1 & 1 & 1& 1& 1& 2 
\end{smallmatrix}, \ 
\begin{smallmatrix}
0 & 0 & 0 & 0& 0& 1& 0 \\
0 & 1 & 1 & 1& 1& 1& 1 \\
0 & 1 & 1 & 1& 1& 1& 1 \\
0 & 1 & 1 & 1& 1& 1& 1 \\
0 & 1 & 1 & 1& 1& 2& 1 \\
1 & 1 & 1 & 1& 2& 1& 2 \\
0 & 1 & 1 & 1& 1& 2& 2 
\end{smallmatrix}, \ 
\begin{smallmatrix}
0 & 0 & 0 & 0& 0& 0& 1 \\
0 & 1 & 1 & 1& 1& 1& 1 \\
0 & 1 & 1 & 1& 1& 1& 1 \\
0 & 1 & 1 & 1& 1& 1& 1 \\
0 & 1 & 1 & 1& 1& 1& 2 \\
0 & 1 & 1 & 1& 1& 2& 2 \\
1 & 1 & 1 & 1& 2& 2& 1
\end{smallmatrix} $$ 
\end{scriptsize}
\subsection{The w-cyclic subfactor planar algebras} \hspace*{1cm} \\
\noindent Let $ \mathcal{P} $ be an irreducible subfactor planar algebra.
\begin{theorem} \label{mini}
Let $ p \in \mathcal{P}_{2,+} $ be a minimal central projection. Then there exists $ v \le p $ minimal projection such that $ \langle v \rangle = \langle p \rangle $.
\end{theorem}
\begin{proof}
 If $p$ is a minimal projection, then it is ok. Else, let $b_1, \dots , b_n$ be the coatoms of $[e_1,\langle p \rangle]$ ($n$ is finite by \cite{wa}). If $p \not \preceq \sum_{i=1}^n b_i$, then $\exists u \le p $ minimal projection such that $u  \not \le b_i \ \forall i$, so that $\langle u \rangle = \langle p \rangle$. Else $p \preceq \sum_{i=1}^n b_i$ (with $n>1$, otherwise $p \le b_1$ and $\langle p \rangle \le b_1$, contradiction). Let $E_i=\text{range}(b_i)$ and $F=\text{range}(p)$, then $F=\sum_i E_i \cap F$ (because $p$ is a minimal central projection) with $1<n<\infty$ and $E_i \cap F \subsetneq F \ \forall i$ (otherwise $\exists i$ with $p \le b_i$, contradiction), so $\dim (E_i \cap F)< \dim (F)$ and there exists $U \subseteq F$ one-dimensional subspace such that $U \not \subseteq E_i \cap F$  $\forall i$, and so a fortiori $U \not \subseteq E_i$   $\forall i$. It follows that $u=p_{U} \le p$  is a minimal projection such that $\langle u \rangle = \langle p \rangle$.
\end{proof} 
Thanks to Theorem \ref{mini}, we can give the following definition: 
\begin{definition} \label{wcy} The planar algebra $ \mathcal{P} $ is weakly cyclic (or w-cyclic) if it satisfies one of the following equivalent assertions: 
\begin{itemize}
\item $ \exists u \in \mathcal{P}_{2,+} $ minimal projection such that $ \langle u \rangle=id $.
\item $ \exists p \in \mathcal{P}_{2,+} $ minimal central projection such that $ \langle p \rangle=id $.
\end{itemize}
We call a subfactor w-cyclic if its planar algebra is w-cyclic.
\end{definition} 
The following remark justifies the choice of the word ``w-cyclic''.
\begin{remark} By Corollary \ref{wgrp2}, a finite group subfactor $ (R^G \subset R) $ is w-cyclic if and only if $ G $ is linearly primitive, which is strictly weaker than cyclic (see for example $ S_3 $), nevertheless the notion of w-cyclic is a singly generated notion in the sense that ``there is a minimal projection generating the identity biprojection''. We can also see the weakness of this assumption by the fact that the minimal projection does not necessarily generate a basis for the set of positive operators, but just the support of it, i.e. the identity. \end{remark} 
\begin{question} Is a cyclic subfactor planar algebra w-cyclic? \\
The answer is \textbf{yes} by Theorem \ref{thm}.
\end{question} 
Let $ \mathcal{P} = \mathcal{P}(N \subseteq M) $ be an irreducible subfactor planar algebra. Take an intermediate subfactor $ N \subseteq K \subseteq M $ and its biprojection $ b = e^M_K $.

\begin{lemma} \label{pos} Let $ \mathcal{A} $ be a $ \star $-subalgebra of $ \mathcal{P}_{2,+} $. Then any element $ x \in \mathcal{A} $ is positive in $ \mathcal{A} $ if and only if it is positive in $ \mathcal{P}_{2,+} $. \end{lemma}
\begin{proof}
If $ x $ is positive in $ \mathcal{A} $, then it is of the form $ aa^{\star} $, with $ a \in \mathcal{A} $, but $ a \in \mathcal{P}_{2,+} $ also, so $ x $ is positive in $ \mathcal{P}_{2,+} $. Conversely, if $ x $ is positive in $ \mathcal{P}_{2,+} $ then $ \langle xy | y \rangle=tr(y^{\star}xy) \ge 0 $, for any $ y \in \mathcal{P}_{2,+} $, so in particular, for any $ y \in \mathcal{A} $, which means that $ x $ is positive in $ \mathcal{A} $.
\end{proof}
\noindent Note that Lemma \ref{pos} will be applied to $ \mathcal{A} = b\mathcal{P}_{2,+}b $ or $ b*\mathcal{P}_{2,+}*b $.

\begin{proposition} \label{left}
The planar algebra $ \mathcal{P}(e_1,b) $ is w-cyclic if and only if there is a minimal projection $ u \in \mathcal{P}_{2,+} $ such that $ \langle u \rangle = b $.
\end{proposition}
\begin{proof}
The planar algebra $ \mathcal{P}(N \subseteq K) $ is w-cyclic if and only if there is a minimal projection $ x \in \mathcal{P}_{2,+}(N \subseteq K) $ such that $ \langle x \rangle = e^K_K $, if and only if $ l_K(\langle x \rangle) = l_K(e^K_K) $, if and only if $ \langle u \rangle = e^M_K $ (by Theorem \ref{2iso}), with $ u=l_K(x) $ a minimal projection in $ e^M_K\mathcal{P}_{2,+}e^M_K $ and in $ \mathcal{P}_{2,+} $. 
\end{proof}

\begin{lemma} \label{lem1}
For any minimal projection $ x \in \mathcal{P}_{2,+}(b,id) $, $ r_b(x) $ is positive and for any minimal projection $ v \preceq r_b(x) $, there is $ \lambda > 0 $ such that $ b * v * b = \lambda r_b(x) $.
\end{lemma}
\begin{proof}
Firstly, $ x $ is positive, so by Theorem \ref{2iso}, $ r_b(x) $ is also positive. For any minimal projection $ v \preceq r_b(x) $, we have $ b * v * b \preceq r_b(x) $, because $$ b * v * b \preceq b*r_b(x)*b = b*b * u * b*b \sim b*u*b = r_b(x), $$ by Lemma \ref{th}(2) and with $ u \in \mathcal{P}_{2,+} $. Now by Lemma \ref{th}(1), $ b*v*b>0 $, so $ r_b^{-1}(b*v*b)>0 $ also, and by Theorem \ref{2iso}, $$ r_b^{-1}(b * v * b) \preceq x. $$ But $ x $ is a minimal projection, so by positivity, $ \exists \lambda > 0 $ such that $$ r_b^{-1}(b * v * b) = \lambda x. $$ It follows that $ b * v * b = \lambda r_b(x) $.
\end{proof}
\begin{lemma} \label{lem2}
Consider $ v \in \mathcal{P}_{2,+} $ positive. Then $ \langle b*v*b \rangle = \langle b,v \rangle $.
\end{lemma}
\begin{proof}
Firstly, by Definition \ref{gener}, $ b * v * b \preceq \langle b,v \rangle $, so by Lemma \ref{th}(3), $ \langle b * v * b \rangle \le \langle b,v \rangle $. Next $ e_1 \le b $ and $ x*e_1 = e_1*x = \delta^{-1}x $, so $$ v = \delta^2 e_1 * v * e_1 \preceq b * v * b. $$ Moreover by Theorem \ref{biproj}, $ \overline{v} \preceq \langle b * v * b \rangle $, but by Lemma \ref{pre2}, $$ \overline{v} * b * v * b \succeq \overline{v} * e_1 * v * b \sim \overline{v} * v * b \succeq e_1 * b \sim b. $$ Then $ b,v \le \langle b * v * b \rangle $, so we also have $ \langle b,v \rangle \le \langle b * v * b \rangle $. 
\end{proof}

\begin{proposition} \label{right}
The planar algebra $ \mathcal{P}(b,id) $ is w-cyclic if and only if there is a minimal projection $ v \in \mathcal{P}_{2,+} $ such that $ \langle b,v \rangle = id $ and $ r_b^{-1}(b*v*b) $ is a positive multiple of a minimal projection.
\end{proposition}
\begin{proof}
The planar algebra $ \mathcal{P}(K \subseteq M) $ is w-cyclic if and only if there is a minimal projection $ x \in \mathcal{P}_{2,+}(K \subseteq M) $ such that $ \langle x \rangle = e^M_M $, if and only if $ r_K(\langle x \rangle) = r_K(e^M_M) $, if and only if $ \langle r_K(x) \rangle = e^M_M $ by Theorem \ref{2iso}. The results follows by Lemmas \ref{lem1} and \ref{lem2}.
\end{proof}

\subsection{The main result} \hspace*{1cm} \\
\noindent Let $ \mathcal{P} $ be an irreducible subfactor planar algebra.
 
\begin{lemma} \label{maximal} A maximal subfactor planar algebra is w-cyclic.
\end{lemma}
\begin{proof}
By maximality $ \langle u \rangle = id $ for any minimal projection $ u \neq e_1 $.
\end{proof} 
 
 \begin{definition} \label{topdef}
The top intermediate subfactor planar algebra is the intermediate associated to the top interval of the biprojection lattice.
\end{definition}
 
\begin{lemma} \label{Topw}
An irreducible subfactor planar algebra is w-cyclic if its top intermediate is so.
\end{lemma}
\begin{proof}
Let $ b_1, \dots , b_n $ be the coatoms in $ [e_1,id] $ and $ t = \bigwedge_{i=1}^n b_i $. By assumption and Proposition \ref{right}, there is a minimal projection $ v \in \mathcal{P}_{2,+} $ with $ \langle t,v \rangle = id $. If $ \exists i $ such that $ v \le b_i $, then $ \langle t,v \rangle \le b_i $, contradiction. So $ \forall i $, $ v \not \le b_i $ and then $ \langle v \rangle = id $. 
\end{proof}
\begin{definition}
Let $ h(\mathcal{P}) $ be the height of the biprojection lattice $[e_1,id]$. Note that $ h(\mathcal{P})<\infty $ because the index is finite.
\end{definition} 
\begin{theorem} \label{centheo}
If the biprojections in $ \mathcal{P}_{2,+} $ are central and form a distributive lattice, then $ \mathcal{P} $ is w-cyclic.
\end{theorem}
\begin{proof} 
By Lemma \ref{c*}, we can make an induction on $ h(\mathcal{P}) $. If $ h(\mathcal{P})=1 $, then we apply Lemma \ref{maximal}. Now suppose that the theorem holds for $ h(\mathcal{P})<n $, we will prove it for $ h(\mathcal{P})=n \ge 2 $. By Lemmas \ref{topBn} and \ref{Topw}, we can assume the biprojection lattice to be boolean. 
For $ b $ in the open interval $(e_1,id)$, its complementary  $ b^{\complement} $ (see Section \ref{baslat}) is also in $(e_1,id)$. By induction and Proposition \ref{left}, there are minimal projections $ u, v $ such that $ b = \langle u \rangle $ and $ b^{\complement} = \langle v \rangle $. Take any minimal projection $ c \preceq u * v $, then $$ \langle c \rangle = \langle c \rangle \vee e_1 = \langle c \rangle \vee (b \wedge b^{\complement}) = \langle c \rangle \vee (\langle u \rangle \wedge \langle v \rangle),$$ so \textit{by distributivity} $$ \langle c \rangle = (\langle c \rangle \vee \langle u \rangle) \wedge (\langle c \rangle \vee \langle v \rangle) = \langle c,u \rangle \wedge \langle c , v \rangle. $$ 
Then by Lemma \ref{pfr}, $ \langle c \rangle = \langle u',c,v \rangle \wedge \langle u,c,v' \rangle $ with $ u', v' $ minimal projections and $ u u', v v' \neq 0 $, so in particular the central support $ Z(u') = Z(u) $ and $ Z(v') = Z(v) $. Now by assumption, every biprojection is central, so $ u \le Z(u') \le \langle u',c,v \rangle $ and $ v \le Z(v') \le \langle u,c,v' \rangle $, then $ \langle c \rangle =id $. \end{proof} 
\begin{theorem} \label{thm}
A cyclic subfactor planar algebra is w-cyclic.
\end{theorem}
\begin{proof} Immediate by Theorem \ref{centheo} because a normal biprojection is by definition bicentral, so a fortiori central.
\end{proof}
\section{Extension for small distributive lattices}
We extend Theorem \ref{centheo} without assuming the biprojections to be central, but for distributive lattices with less than $ 32 $ elements.
Because the top lattice of a distributive lattice is boolean (Lemma \ref{topBn}), we can reduce the proof to $ \mathcal{B}_n $ with $ n<5 $.
\begin{definition}
An irreducible subfactor planar algebra is said to be boolean (or $ \mathcal{B}_n $) if its biprojection lattice is boolean (of rank $n$).
\end{definition}
\begin{proposition} \label{l<} An irreducible subfactor planar algebra such that the coatoms $ b_1, \dots, b_n \in [e_1,id] $ satisfy $ \sum_i \frac{1}{|id:b_i|} \le 1 $, is w-cyclic.
 \end{proposition}
 \begin{proof} Firstly, by Lemmas \ref{maximal} and \ref{Topw}, we can assume that $ n>1 $.  By Definition $ |id:b_i| = \frac{tr(id)}{tr(b_i)} $ so by assumption $ \sum_i tr(b_i) \le tr(id) $.  If $ \sum_i b_i \sim id $ then $ \sum_i b_i \ge id $, but $ \sum_i tr(b_i) \le tr(id) $ so $ \sum_i b_i = id $. Now $ \forall i \ e_1 \le b_i $, so $ ne_1 \le \sum_i b_i = id $, contradiction with $ n>1 $.  So $ \sum_i b_i \prec id $, which implies the existence of a minimal projection $ u \not \le b_i \ \forall i $, which means that $ \langle u \rangle = id $. 
\end{proof}
\begin{remark} The converse is false, $ (R \subset R \rtimes \mathbb{Z}/30) $ is a counter-example, because $ 1/2+1/3+1/5 = 31/30 >1 $. \end{remark}
\begin{corollary} \label{two} An irreducible subfactor planar algebra with at most two coatoms in $ [e_1,id] $ is w-cyclic.
 \end{corollary}
\begin{proof} $ \sum_i \frac{1}{ |id:b_i| } \le 1/2+1/2 $, the result follows by Proposition \ref{l<}. \end{proof}
\begin{examples} Every $ \mathcal{B}_2 $ subfactor planar algebra is w-cyclic. 
 $$ 
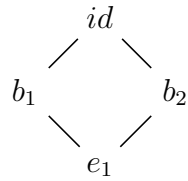
\begin{tikzpicture}
\node (A1) at (0,0) { $ id $ };
\node (A2) at (-1,-1) { $ b_1 $ };
\node (A4) at (1,-1) { $ b_2 $ };
\node (A5) at (0,-2) { $ e_1 $ };
\tikzstyle{segm}=[-,>=latex, semithick]
\draw [segm] (A1)to(A2); \draw [segm] (A1)to(A4);
\draw [segm] (A2)to(A5);\draw [segm] (A4)to(A5);
\end{tikzpicture} $$ 
 \end{examples}

\begin{lemma} \label{wmin} 
Let $ u,v \in \mathcal{P}_{2,+} $ be minimal projections. If $ v \not \le \langle u \rangle $ then $ \exists c \preceq u * v $ and $ \exists w \preceq \overline{u} * c $ minimal projections such that $ w \not \le \langle u \rangle $.
\end{lemma}
\begin{proof} Assume that $ \forall c \preceq u * v $ and $ \forall w \preceq \overline{u} * c $ we have $ w \le \langle u \rangle $. Now there are minimal projections $ (c_i)_i $ and $ (w_{i,j})_{i,j} $ such that $ u * v \sim \sum_i c_i $ and $ \overline{u} * c_i \sim \sum_j w_{i , j} $. It follows that $ u * v \sim \sum_{i,j} w_{i,j} \preceq \langle u \rangle $, but 
 $$ v \sim e_1 * v \preceq (\overline{u} * u)*v = \overline{u} * (u*v) \preceq \langle u \rangle, $$ 
which is in contradiction with $ v \not \le \langle u \rangle $. \end{proof} 
 
For the distributive case, we can upgrade Proposition \ref{l<} as follows:
\begin{theorem} \label{le2}
A distributive subfactor planar algebra such that the coatoms $ b_1, \dots, b_n \in [e_1,id] $ satisfy $ \sum_i \frac{1}{|id:b_i|} \le 2 $, is w-cyclic.
\end{theorem}
\begin{proof}
By Lemmas \ref{topBn} and \ref{Topw}, we can assume the subfactor planar algebra to be boolean. \\
If $ K:=\bigwedge_{i,j, i \neq j} (b_i \wedge b_j)^{\perp} \neq 0 $, then consider $ u \le K $ a minimal projection, and $ Z(u) $ its central support. If $ \langle Z(u) \rangle = id $, then we are ok. Else $ \exists i $ such that $ \langle u \rangle = \langle Z(u) \rangle = b_i $. But $ b_i^{\complement} $ is an atom in $ [e_1,id] $, so there is a minimal projection $v$ such that $ b_i^{\complement} = \langle v \rangle $. Recall that $ b_i \wedge b_i^{\complement} = e_1 $, so $ v \not \le \langle u \rangle $, and by Lemma \ref{wmin}, there are minimal projections $ c \preceq u * v $ and $ w \preceq \overline{u} * c $ such that $ w \not \preceq \langle u \rangle $ (and $ \langle u,w \rangle = id $ by maximality). By Lemma \ref{pfr}, $ \exists u' \preceq c * \overline{u} $ with $ Z(u') = Z(u) $ and $ u' \not \perp u $, but $ u \le K $ so $ \forall j \neq i $, $ u' \not \le b_i \wedge b_j $, now $ u' \le Z(u) \le b_i $, so $ \langle u' \rangle = b_i $. Using distributivity (as for Theorem \ref{centheo}) we conclude by $$ \langle c \rangle = \langle u,c \rangle \wedge \langle c,v \rangle \ge \langle u,w \rangle \wedge \langle u',v \rangle = id \wedge id = id. $$ 
Else $ K = 0 $, but $ \forall i $, $ (b_i \wedge b_j)^{\perp} \ge b_j^{\perp} $, so $ \forall i $, $ \bigwedge_{j \neq i} b_j^{\perp} = 0 $. Let $ p_1, \dots , p_r $ be the minimal central projections. Then $ b_i = \bigoplus_{s=1}^{r} p_{i,s} $ with $ p_{i,s} \le p_s $ and $ p_{i,1} = p_1 = e_1 $. Now $ b_i^{\perp} = \bigoplus_{s=1}^{r} (p_s-p_{i,s}) $, so by assumption, $$ 0 = \bigwedge_{j \neq i} \bigoplus_{s=1}^{r} (p_s-p_{j,s}) = \bigoplus_{s=1}^{r} \bigwedge_{j \neq i} (p_s-p_{j,s}), \ \forall i. $$ It follows that for all $i$ and $s$, $ p_s = \bigvee_{j \neq i} p_{j,s} $, so $ tr(p_s) \le \sum_{j \neq i} tr(p_{j,s}) $. Now if $ \exists s$ such that $ \forall i \ p_{i,s} < p_s $, then $ \langle p_s \rangle = id $, which is ok; else $ \forall s, \exists i $ with $ \ p_{i,s} = p_s $, but $ \sum_{j \neq i} tr(p_{j,s}) \ge tr(p_s) $, so $ \sum_{j} tr(p_{j,s}) \ge 2tr(p_s) $. Then $$ \sum_i tr(b_i) \ge n \cdot tr(e_1) + 2 \sum_{s \neq 1} tr(p_s) = 2tr(id) + (n-2)tr(e_1). $$ 
Now $ |id:b_i| = tr(id)/tr(b_i) $, so $$ \sum_i \frac{1}{|id:b_i|} \ge 2 + \frac{n-2}{|id:e_1|} $$ 
which contradicts the assumption, because we can assume $ n>2 $ by Corollary \ref{two}. The result follows.
\end{proof}
\begin{remark} The converse is false because there exists w-cyclic distributive subfactor planar algebras with $ \sum_i \frac{1}{|id:b_i|}>2 $. For example, the subfactor $ (R \rtimes S^n_2 \subset R \rtimes S^n_3) $ is w-cyclic and $ \mathcal{B}_n $, but $ \sum_i \frac{1}{|id:b_i|} = n/3 $.
\end{remark}
\begin{corollary} \label{n/2} \label{B4w}
Every $ \mathcal{B}_n $ subfactor planar algebra with $ |id:b| \ge n/2 $, for any coatom $ b \in [e_1,id] $, is w-cyclic. Then $\forall n \le 4 $, any $ \mathcal{B}_n $ subfactor planar algebra is w-cyclic.
\end{corollary}
\begin{proof}
By assumption (following the notations of Theorem \ref{le2}) $$ \sum_i \frac{1}{|id:b_i|} \le \sum_i \frac{2}{n} = 2. $$ But $ |id:b| \ge 2 $, so any $ n \le 4 $ works.
\end{proof}
\begin{corollary} \label{32}
A distributive subfactor planar algebra having less than $32$ biprojections (or of index $ <32 $), is w-cyclic. 
\end{corollary}
\begin{proof}
In this case, the top of $[e_1,id]$ is boolean of rank $ n < 5 $, because $ 32 = 2^5 $; the result follows by Lemma \ref{Topw} and Corollary \ref{B4w}.
\end{proof}
\begin{conjecture} \label{conjext}
A distributive subfactor planar algebra is w-cyclic.
\end{conjecture}
\noindent By Lemmas \ref{topBn} and \ref{Topw}, we can reduce Conjecture \ref{conjext} to the boolean case, and then extend it to the top boolean case. 
\begin{remark} The converse of Conjecture \ref{conjext} is false, because the group $ S_3 $ is linearly primitive but not cyclic (see Corollary \ref{wgrp2}).
\end{remark}
\begin{problem} What is the natural additional assumption (A) such that $ \mathcal{P} $ is distributive if and only if it is w-cyclic and satisfies (A)?
\end{problem}
\noindent Assuming Conjecture \ref{conjext} and using Remark \ref{distable}, we get:
\begin{conjecture} \label{statw*} For any distributive subfactor planar algebra $ \mathcal{P} $ and any biprojection $ b \in \mathcal{P}_{2,+} $, the planar algebras $\mathcal{P}(e_1,b)$, $\mathcal{P}(b,id)$ and their duals are w-cyclic.
\end{conjecture}
\begin{remark} The converse is false because the interval $[S_2,S_4]$, proposed by Zhengwei Liu, gives a counter-example.
\end{remark}
\begin{remark} A cyclic subfactor planar algebra satisfies Conjecture \ref{statw*} (thanks to Theorem \ref{thm} and Lemma \ref{c*}).
\end{remark}
\begin{problem} Is a Dedekind subfactor planar algebra $ \mathcal{P} $ distributive if and only if for any biprojection $ b \in \mathcal{P}_{2,+} $, the planar algebras $\mathcal{P}(e_1,b)$, $\mathcal{P}(b,id)$ and their duals are w-cyclic?
\end{problem}
 
\section{Applications}
\subsection{A non-trivial upper bound}
For any irreducible subfactor planar algebra $ \mathcal{P} $, we exhibit a non-trivial upper bound for the minimal number of minimal $ 2 $-box projections generating the identity biprojection. We will use the notations of Section \ref{interpa}.
\begin{lemma} \label{interdistrib}
Let $ b' < b $ be biprojections. If $ \mathcal{P}(b' , b) $ is w-cyclic, then there is a minimal projection $ u \in \mathcal{P}_{2,+} $ such that $ \langle b', u \rangle = b $. 
\end{lemma}
\begin{proof} Consider the isomorphisms of von Neumann algebras $$ l_{b}: \mathcal{P}_{2,+}(e_1 , b) \to b\mathcal{P}_{2,+}b $$ and, with $ a = l_{b}^{-1}(b') $, $$ r_{a}: \mathcal{P}_{2,+}(b' , b) \to a * \mathcal{P}_{2,+}(e_1 , b) * a. $$ Then, by assumption, the planar algebra $ \mathcal{P}(b' , b) $ is w-cyclic, so by Proposition \ref{right}, $ \exists u' \in \mathcal{P}_{2,+}(e_1,b) $ minimal projection such that $$ \langle a, u' \rangle = l_b^{-1}(b). $$ Then by applying the map $ l_{b} $ and Theorem \ref{2iso}, we get $$ b=\langle l_{b}(a), l_{b}(u') \rangle = \langle b', u \rangle $$ with $ u=l_{b}(u') $ a minimal projection in $ b\mathcal{P}_{2,+}b $, so in $ \mathcal{P}_{2,+} $.
\end{proof}
Assuming Conjecture \ref{conjext} and using Lemma \ref{interdistrib}, we get a non-trivial upper bound:
 
\begin{conjecture} \label{statement}
The minimal number $ r $ of minimal projections generating the identity biprojection (i.e., $ \langle u_1, \dots, u_r \rangle $ = id) is less than the minimal length $ \ell $ for an ordered chain of biprojections $$ e_1=b_0 < b_1 < \dots < b_{\ell} = id $$ such that $ [b_i, b_{i+1}] $ is distributive (or better, top boolean).
\end{conjecture} 
\begin{remark} We can deduce theorems from Conjecture \ref{statement}, by adding some assumptions to $ [b_i, b_{i+1}] $, according to Theorems \ref{centheo} or \ref{le2}. \end{remark}
\begin{remark} Let $ (N \subset M) $ be any irreducible finite index subfactor. We can deduce a non-trivial upper bound for the minimal number of (algebraic) irreducible sub-$ N $-$ N $-bimodules of $ M $, generating $ M $ as von Neumann algebra.
\end{remark} 
\subsection{Back to the finite groups theory}
As applications, we get a dual version of Theorem \ref{ore2}, and for any finite group $ G $, we get a non-trivial upper bound for the minimal number of irreducible components for a faithful complex representation. The action of $G$ on the hyperfinite $ {\rm II}_1 $ factor $ R $ is always assumed outer.

\begin{theorem}[\S 226 \cite{bu}] \label{faith}
A complex representation $V$ of a finite group $G$ is faithful if and only if any irreducible complex representation $W$ is equivalent to a subrepresentation of  $V^{\otimes n} $, for some $n \ge 0$.
\end{theorem}

\begin{definition} \label{linprimdef}
A group $ G $ is linearly primitive if it admits a faithful irreducible complex representation.
\end{definition}
 \begin{definition} \label{fixstab} Let $ W $ be a representation of a group $ G $, $ K $ a subgroup of $ G $, and $ X $ a subspace of $ W $. Let the \textit{fixed-point subspace} be $$ W^{K}:=\{w \in W \ \vert \ kw=w \ , \forall k \in K \} $$ and the \textit{pointwise stabilizer subgroup} $$ G_{(X)}:=\{ g \in G \ \vert \ gx=x \ , \forall x \in X \} $$ \end{definition} 
\begin{definition} \label{linprim}
An interval $ [H,G] $ is said to be linearly primitive if there is an irreducible complex representation $ V $ of $ G $ with $ G_{(V^H)} = H $.
\end{definition}
\noindent The group $ G $ is linearly primitive iff the interval $ [\{e\},G] $ is so. 
\begin{lemma} \label{linprimgrp} Let $H$ be a core-free subgroup of G. Then $ G $ is linearly primitive if $ [H,G] $ is so. 
\end{lemma}
\begin{proof} Take $ V $ as above. Now, $ V^H \subset V $ so $ G_{(V)} \subset G_{(V^H)} $, but $ \ker(\pi_V) = G_{(V)} $, it follows that $ \ker(\pi_V) \subset H $; but $ H $ is a core-free subgroup of $ G $, and $ \ker(\pi_V) $ a normal subgroup of $ G $, so $ \ker(\pi_V)= \{ e \} $, which means that $ V $ is faithful on $ G $, i.e. $ G $ is linearly primitive. \end{proof}
\begin{lemma} \label{corrminstab}
Let $ p_x \in \mathcal{P}_{2,+}(R^G \subseteq R) $ be a minimal projection on the one-dimensional subspace $ \mathbb{C}x $ and $ H $ a subgroup of $ G $. Then $$ p_x \le b_H:=\vert H \vert^{-1}\sum_{h \in H} \pi_{V}(h) 
\Leftrightarrow H \subset G_x. $$ 
\end{lemma}
\begin{proof} If $ p_x \le b_H $ then $ b_H x = x $ and 
 $ \forall h \in H $ we have that $$ \pi_V(h) x = \pi_V(h) [b_H x] = [\pi_V(h) \cdot b_H] x = b_H x = x $$ 
which means that $ h \in G_x $, and so $ H \subset G_x $. 
Conversely, if $ H \subset G_x $ (i.e. $ \forall h \in H $, $ \pi_V(h) x = x $) then $ b_H x = x $, which means that $ p_x \le b_H $.
\end{proof}
\begin{theorem} \label{wgrp}
Let $[H,G]$ be an interval of finite groups. Then
\begin{itemize}
\item $ (R \rtimes H \subseteq R \rtimes G) $ is w-cyclic if and only if $ [H,G] $ is $ H $-cyclic. 
\item $ (R^G \subseteq R^H) $ is w-cyclic if and only if $ [H,G] $ is linearly primitive.
\end{itemize} 
\end{theorem}
\begin{proof}
By Proposition \ref{right}, $ (R \rtimes H \subseteq R \rtimes G) $ is w-cyclic if and only if $$ \exists u \in \mathcal{P}_{2,+}(R \subseteq R \rtimes G) \simeq \bigoplus_{g \in G} \mathbb{C}e_g \simeq \mathbb{C}^G $$ minimal projection such that $ \langle b , u \rangle = id $, with $ b=e^{R \rtimes G}_{R \rtimes H} $ and $ r_b^{-1}(b*u*b) $ is a minimal projection; if and only if $ \exists g \in G $ such that $ \langle H,g \rangle = G $, because $ u $ is of the form $ e_g $ and $ \forall g' \in HgH $, $ Hg'H = HgH $. 

 By Proposition \ref{left}, $ (R^G \subseteq R^H) $ is w-cyclic if and only if $$ \exists u \in \mathcal{P}_{2,+}(R^G \subseteq R) \simeq \bigoplus_{V_i \ irr.}\text{End}(V_i) \simeq \mathbb{C}G $$ minimal projection such that $ \langle u \rangle = e^R_{R^H} $; if and only if, by Lemma \ref{corrminstab}, $ H=G_x $ with $ u=p_x $ the projection on $ \mathbb{C}x \subseteq V_i $ (with $ Z(p_x) = p_{V_i} $). \\ 
Note that $ H \subset G_{(V_i^H)} \subset G_x $ so $ H = G_{(V_i^H)} $. \end{proof}
\begin{corollary} \label{wgrp2} The subfactor $ (R^G \subseteq R) $ (resp. $ (R \subseteq R \rtimes G) $) is w-cyclic if and only if $ G $ is linearly primitive (resp. cyclic). \end{corollary} 
\begin{examples} The subfactors $ (R^{S_4} \subset R^{S_2}) $, its dual and $ (R^{S_3} \subset R) $,  are w-cyclic, but $ (R \subset R \rtimes S_3) $ and $ (R^{S_4} \subset R^{\langle (1,2)(3,4) \rangle}) $ are not. \end{examples} 
\noindent By Theorem \ref{wgrp}, the group-theoretic reformulation of Conjecture \ref{conjext} on $ (R^G \subseteq R^H) $ is the following dual version of Theorem \ref{ore2}.
\begin{conjecture} \label{statgroup}
Let $ [H,G] $ be a distributive interval of finite groups. Then $ \exists V $ irreducible complex representation of $ G $ such that $ G_{(V^H)} = H $.
\end{conjecture}
\noindent If moreover $ H $ is core-free, then $ G $ is linearly primitive (Lemma \ref{linprimgrp}).
\begin{problem} Is a finite group $ G $ linearly primitive iff there is a core-free subgroup $ H $ such that the interval $ [H,G] $ is bottom boolean? 
\end{problem}
\noindent By Theorem \ref{faith}, Conjecture \ref{statement} on $ \mathcal{P}(R^G \subseteq R) $ reformulates as follows:
\begin{conjecture} \label{statdualgroup}
The minimal number of irreducible components for a faithful complex representation of a finite group $ G $ is less than the minimal length $ \ell $ for an ordered chain of subgroups $$ \{e\}=H_0 < H_1 < \dots < H_{\ell} = G $$ such that $ [H_i,H_{i+1}] $ is distributive (or better, bottom boolean).
\end{conjecture}
\noindent This provides a bridge linking combinatorics and representations in the theory of finite groups.
\begin{remark} We can upgrade Conjecture \ref{statdualgroup} by taking for $ H_0 $ any core-free subgroup of $H_1$, instead of just $\{e\}$; we can also deduce theorems, by adding some assumptions to $ [H_i, H_{i+1}] $, according to the group-theoretic reformulation of Theorems \ref{thm} or \ref{le2}. Note that a normal biprojection in $ \mathcal{P}(R^G \subseteq R^H) $ is given by a subgroup $ K \in [H,G] $ with $ HgK=KgH \ \forall g \in G $, see \cite[Proposition 3.3]{teru}. \end{remark}
\begin{remark} We can also formulate results for finite quantum groups (i.e. finite dimensional Kac algebras), where the biprojections correspond to the left coideal $ \star $-subalgebras, see \cite[Theorem 4.4]{ilp}. \end{remark}
\section{Acknowledgments} 
This work is supported by the Institute of Mathematical Sciences, Chennai. The author is grateful to Vaughan Jones, Dietmar Bisch, Scott Morrison and David Evans for their recommendation for a postdoc at the IMSc. Thanks to my hosts V.S. Sunder and Vijay Kodiyalam, and to Zhengwei Liu, Feng Xu, Keshab Chandra Bakshi and Mamta Balodi, for useful advices and fruitful exchanges.

\end{document}